\documentclass[11pt, reqno]{article}

\usepackage{amssymb}
\usepackage{amsmath}
\usepackage{amsthm}
\usepackage{amscd}
\usepackage{cancel}
\usepackage{float}
\usepackage{graphicx}
\usepackage{amsfonts}
\usepackage{color}
\usepackage{xypic}
\usepackage{enumerate}
\usepackage{lipsum}
\usepackage{mathtools}
\usepackage[all]{xy}

\usepackage{hyperref}
\usepackage{xcolor}
\hypersetup{
    colorlinks,
    linkcolor={red!50!black},
    citecolor={blue!50!black},
    urlcolor={blue!80!black}
}

\usepackage{amsfonts}
\usepackage{cancel}
\usepackage[normalem]{ulem}
\newtheorem{conjecture}{Conjecture}
\newtheorem{theorem}{Theorem}
\newtheorem{corollary}{Corollary}

\newtheorem{proposition}{Proposition}

\theoremstyle{definition}
\newtheorem{definition}{Definition}
\newtheorem{remark}{Remark}

\newcommand{\KK}{\mathbb{K}}
\newcommand{\RR}{\mathbb{R}}

\newcommand{\tr}{{\rm tr}}
\newcommand{\T}{{\rm Tr}}
\newcommand{\B}{\mathtt{B}}
\newcommand{\A}{\mathsf{a}}
\newcommand{\bb}{\mathsf{b}}
\newcommand{\Tr}[1]{\mathsf{tr}_{#1}}

\newcommand{\E}{\mathcal{E}}
\newcommand{\EE}[2]{\mathcal{E}_n(#1,#2)}
\newcommand{\SE}[2]{\mathcal{S}_{#1}\mathcal{E}_{#2}}
\newcommand{\SM}[1]{\mathfrak{S}_{#1}}
\newcommand{\s}[2]{\mathcal{S}_{#1}\mathcal{E}_{#2}}
\newcommand{\TS}[2]{TS_{#1}B_{#2}}
\newcommand{\g}[1]{g_{#1}}
\newcommand{\n}{{\mathbf{n}}}
\newcommand{\PP}[1]{\mathsf{P}_{#1}}

\newcommand{\LL}{\mathbb{L}}

\begin{document}

\title{Invariants for singular links via the two parameter bt--algebra}

 \author{Marcelo Flores, Christopher Roque--Marquez }
\date{}
\maketitle
\begin{abstract}
We construct a new invariant of singular links through representations of the singular braid monoid into the two parameters bt--algebra. Additionally, we recover this invariant by using the approach of Paris and Rabenda. Hence, we introduce the so called two parameter Singular bt--algebra. Finally, we provide the skein relations that define our invariant, and we prove that this invariant is more powerful than previous invariants of singular links in literature.
\end{abstract}



\section{Introduction}

The developments in knot theory grew up significantly since the discovery of the Jones polynomial in the 80's. Jones' original approach \cite{Jo85} can be extended for constructing invariants of more general knot-like objects through the following main ingredients: a braided monoid $M$ that is related with the knot-like objects of interest via analogues of Alexander and Markov theorems; an algebra $A$ together with a representation from $M$ into $A$; and lastly, a trace function defined on the algebra $A$ that is compatible with the relations established by the corresponding Markov theorem, which is  commonly called a Markov trace. An algebra that can be used in this context is called knot algebra, cf. \cite{jula7}. In \cite{Jo85,Jo87}, Jones uses the Artin braid group $B_n$ and the Temperley-Lieb algebra as the corresponding knot algebra to obtain the Jones polynomial, while the HOMFLY-PT polynomial is obtained by using the Iwahori-Hecke algebra .\smallbreak
 
In this paper we will focus on singular knots, which arose firstly in the context of Vassiliev invariants, one of the most general knot invariants in the sense that many of them can be deduced from Vassiliev invariants \cite{BirLin93,Vas90}. A singular link on $n$ components is the image of a smooth immersion of $n$ copies of the circle $S^1$ into $S^3$, that has finitely many transversal double points, which are called singular crossings. Thus, a classical link can be regarded as a singular link with no singular crossings. Hence, singular links generalize classical links in $S^3$. Two singular links are isotopic (or equivalent) if there is an orientation preserving self-homeomorphism of $S^3$ such that maps one into the other that fixes a small rigid disc around each singular crossing. \smallbreak 
The algebraic counterpart of singular links is the singular braid  monoid $SB_n$ \cite{Baez92,bir93}. This  monoid plays the role that the Artin's braid group $B_n$ plays for classical links; that is, there exist analogues of Alexander \cite{bir93} and Markov theorems \cite{ge2005} for singular links that relate singular links with the monoid $SB_n$. This allows to construct singular link invariants by using the Jones method, that is by using an algebra that supports a Markov trace and a proper representation from the singular braid monoid $SB_n$ into such algebra. \smallbreak 
In \cite{pare}, Paris and Rabenda introduce the singular Hecke algebra, denoted by $SH_n:=SH_n(u)$, as the quotient of the monoid algebra $\mathbb{C}(u)[SB_n]$ over the ideal generated by the quadratic relations of the Iwahori--Hecke algebra, that is,  $\sigma_i^2-(u-1)\sigma_i -u $, for $1\leq i\leq n-1$. For $d\geq 0$, let $S_dB_n$ be the set of singular braids with $d$ singular crossings. Then, we have the natural graduation
\begin{equation}\label{grad}
 SH_n=\bigoplus_{d=0}^{\infty}S_dH_n,
\end{equation}
where $S_dH_n$ denotes the subspace of $SH_n$ spanned by $S_dB_n$. The algebra $SH_n$ is infinite dimensional, however each $S_dH_n$ is finite dimensional over $\mathbb{C}$. Using this fact, they construct a Markov trace on the sequence $\{S_dH_n\}_{n=1}^{\infty}$, from which the HOMFLY--PT analogue for singular links is derived, which is denoted by $\widehat{I}$. Moreover, this invariant recovers the Kauffman-Vogel invariant \cite{KauVog92} as a specialization, for details see \cite[Lemma 5.6]{pare}.    \smallbreak

In \cite{aiju}, Aicardi and Juyumaya introduce the algebra of braids and ties (or bt--algebra), denoted by $\mathcal{E}_n:=\mathcal{E}_n(u)$. This algebra is defined by abstractly considering it as the subalgebra of the Yokonuma--Hecke algebra generated by the braid generators and the family of idempotents appearing in the quadratic relations of such generators. In \cite{aijuMMJ}, it is proved that $\mathcal{E}_n$ supports a Markov trace; the invariant for classical links $\Delta$ is derived by using Jones method. It is worth noting that the invariant $\Delta$ coincides with the HOMFLY--PT polynomial on knots and it distinguishes pairs of links that are  not distinguished by the HOMFLY--PT polynomial, see \cite[Section7.2]{aijuMMJ}. In \cite{chporep}, the authors give a presentation for the Yokonuma--Hecke algebra with a different quadratic relation, which is obtained by re-scaling the generators, for details see \cite[Remark 1]{chporep}. Consequently, the bt--algebra can also be presented with such quadratic relation by applying the same argument, we denote the bt--algebra with this presentation by $\mathcal{E}_n(\sqrt{u})$. Surprisingly, the classical link invariant derived using this presentation, denoted by $\Theta$, is not equivalent to $\Delta$, \cite[Section 7.2]{aijuMMJ}. This fact motivates the definition of the two parameter bt algebra \cite{aiju2020}, denoted by $\mathcal{E}_n(u,v)$, which includes a quadratic relation that depends of two parameters in its presentation. This algebra is isomorphic to the bt--algebra and therefore it supports a Markov trace, see \cite[Proposition 1]{aiju2020}.  Consequently, the invariant for classical links $\Upsilon$ is derived. It is worth noting that the invariant $\Upsilon$ distinguishes pairs of links that are not distinguished by $\Delta$ and either by $\Theta$, for details see \cite[Section 6]{aiju2020}. \smallbreak

In \cite{aiju2}, the authors introduce the concept of tied links and tied braids, which generalize to classical links and classical braids, respectively. The tied braid monoid, denoted by $TB_n$ is the algebraic counterpart of tied links. That is, there exists analogues of Alexander and Markov theorems that relate tied links and tied braids, see \cite[Theorem 3.5]{aiju2} and \cite[Theorem 3.7]{aiju2}, respectively. Additionally, we have that the bt--algebra can be regarded as a quotient of monoid algebra $\mathbb{C}(\sqrt{u})[TB_n]$. Consequently, the authors obtain an invariant of tied links by applying the Jones method to the bt--algebra $\mathcal{E}_n$, denoted by $\widetilde{\Delta}$, which extends the classical link invariant $\Delta$. Analogously, the invariants $\Theta$ and $\Upsilon$ can be extended to the tied links invariants $\widetilde{\Theta}$ and $\widetilde{\Upsilon}$ by using the algebras $\mathcal{E}_n(\sqrt{u})$ and $\mathcal{E}_n(u,v)$, respectively.


 In \cite{aiju2018}, four invariants for singular links are defined, denoted by $\Psi_{x,y}, \Phi_{x,y}, \Psi'_{x,y}, \Phi'_{x,y}$. All of them are derived by using Jones method on the Markov trace on the bt--algebra $\E_n(u)$, though using different homomorphism from the monoid $SB_n$ to $\E_n(u)$ in each case. Moreover, it is proved that these invariants distinguish pairs of singular links that are not distinguished by the Paris--Rabenda invariant $\widehat{I}$, see \cite[Theorem 16]{aiju2018}. Additionally, it is proved that the tied braid monoid $TB_n$ can be regarded as the semidirect product $\PP{n}\rtimes B_n$, where $\PP{n}$ is the monoid of the set partitions of $\{1,\ldots,n\}$. This fact motivates the introduction of the concept of ct--link (combinatorial tied link) that formalize in some way the concept of tied links previously introduced in \cite{aiju2}. Additionally, the authors also introduce the tied singular braid monoid $TSB_n$ as the semidirect product of $\PP{n}$ and the singular braid monoid $SB_n$ \cite{bir93}. Consequently, the authors introduce the concept of cts--link (combinatorial singular tied link) and the invariants $\Psi_{x,y}, \Phi_{x,y}, \Psi'_{x,y}, \Phi'_{x,y}$ are extended to invariant of cts--links, and skein rules are given for such extensions. \smallbreak

 In \cite{arju2020}, Arcis and Juyumaya, motivated by the fact that $TB_n$ can be regarded as semidirect product, introduce the concept of tied monoids. These monoids are defined as semidirect products that are commonly built from braid groups and their underlying Coxeter group acting on monoids of set partitions. Additionally, they define a tied algebra associated to each tied monoid as quotient of the corresponding monoid algebra. The bt--algebras of type $\mathtt{A}$ and $\mathtt{B}$ \cite{fl} are particular cases of these algebras, which are associated to the monoids $TB_n$ and $TB_n^{\B}$, respectively. Similarly, the tied singular braid monoid $TSB_n$ can be considered as a tied monoid. Hence, we can obtain a tied algebra associated to $TSB_n$, which is introduced in Section \ref{sec:invariant singular bt}. \smallbreak
 
Thus, the behavior of the classical links invariants $\Delta$, $\Theta$ and $\Upsilon$ referenced above suggests that a singular link invariant constructed via representations of the monoid $SB_n$ into the two parameter bt--algebra $\mathcal{E}_n(u,v)$ has to distinguish pairs of singular links that the invariants obtained in \cite{aiju2018} do not. Additionally, it is also able to construct a singular link invariant using the Paris and Rabenda approach though this time using a tied algebra associated to the monoid $TSB_n$. The invariant obtained by this process should be powerful than $\widehat{I}$ since the invariants derived from tied algebras are powerful than the HOMFLY--PT polynomial. In this work, we obtain singular links invariants derived from these two processes and we prove that both of them are equivalent.

The article is organized as follows. In Section~\ref{pre}, we provide the notation and some basic results, and we introduce the algebras and monoids (groups) that we will use along this work. In Section~\ref{newinv}, we construct two singular links invariants by using two different representation from $SB_n$ into the two parameter bt--algebra $\E_n(u,v)$. Additionally, we extend these invariants to invariants of cts--links and we give skein rules for them. Finally, we compare these invariants with singular links invariants defined previously in \cite{aiju2018} and \cite{pare}. In Section~\ref{singularbt}, we introduce the two parameter Singular bt--algebra, denoted by  $\SE{}{n}:=\SE{}{n}(u,v)$, which emerges as a tied algebra associated to the monoid $TSB_n$ introduced in \cite{aiju2018}. We prove that this algebra supports a Markov trace by using the Paris and Rabenda approach. Consequently, we derive an invariant of cts--links by applying Jones method. Finally, we prove that such invariant is equivalent to the obtained in Section~\ref{newinv}.

\section{Preliminaries}\label{pre}
In this section, we recall the main results and we fix notation that will be used in the sequel.
\subsection{Set partitions}
Let $\PP{n}$ be the set formed by the set-partitions of $\n:=\{1,2, \ldots,n\}$.
    The subsets of $\n$ entering a partition are called blocks. For convenience, we will omit the subsets of cardinality one (single blocks) in the partition. For instance, the partition $I=(\{1,2,3\},\{4,6\},\{5\},\{7\})$ in $\PP{7}$ will be simply written as $I=(\{1,2,3\},\{4,6\})$. Recall that the cardinality of $\PP{n}$ is the $n$-th Bell number, denoted by $b_n$.
    The elements of $\PP{n}$ can be represented by its scheme, where every block is represented by 
    a connected component of its elements joined by arcs as follows (see [16, Subsection 3.2.4.3]), the point $i$ is connected by an arc to the point $j$ if both are in the same block, $i<j$, and there is no $k$ in the same block with $i<k<j$. In Fig. \ref{scheme} a set partition is represented by arcs. \smallbreak

\begin{figure}[h!]
        \centering
         \includegraphics[scale=1]{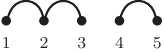}\label{scheme}
        \caption{Scheme of the partition $(\{1,2,3\},\{4,5\}).$}
\end{figure}

The symmetric group $\mathfrak{S}_n$ acts naturally on $\PP{n}$ by permuting each element of every block of a partition.
\smallbreak

For $I=(I_1,\dots,I_m),\ J=(J_1,\dots,J_s) \in \PP{n}$, we say that $I$ refines $J$ if every block of $J$ is a union of some blocks of $I$, and we denote it by $I\preceq J$. Then $\preceq$ defines a partial order on $\PP{n}$. Observe that the empty partition, denoted by $\mathbf{1}_n$, is the least element of $\PP{n}$ respect to this order.\smallbreak

We can define the following product in $\PP{n}$
$$IJ:={\rm min}\{K\in \PP{n} \ |\ I\preceq K \wedge J\preceq K\}.$$
Thus, $\PP{n}$ has structure of commutative monoid with this product. More precisely, we have that the monoid $\PP{n}$ is presented by generators $\mu_{i,j}$ for all $i,j\in\n$ with $i < j$ subject to the following relations:

\begin{equation}\label{rel part}
\begin{array}{rcl}
  \mu_{i,j}^2   &= & \mu_{i,j}\quad \text{for all $i<j$,} \\
    \mu_{i,j}\mu_{k,l} &= & \mu_{k,l}\mu_{i,j}\quad \text{for all $i<j$ and $k<l$,}\\
    \mu_{i,j}\mu_{j,k} &=&\mu_{i,k}\mu_{j,k}=\mu_{i,k}\mu_{i,j} \quad \text{for all $i<j<k$,}
\end{array}
\end{equation}
where $\mu_{i,j}$ denotes the partition of $\n$ whose unique non single block is $\{i,j\}$ for $i<j$, see \cite{Fitz03}.

\subsection{Singular links and singular braids}
A \textit{singular link on $k$ components} is the image of a smooth immersion of $k$ copies of
the circle $S^1$ into $S^3$, that has finitely many transversal double points, which are called \textit{singular crossings}. So, a singular link is like a classical link, but with a finite number of transversal self-intersections. Two singular links are isotopic (or equivalent) if there is an orientation preserving self-homeomorphism of $S^3$ such that maps one into the other, and which leaves a small rigid disc around each singular crossing unaltered. Thus, a classical link can be regarded as a singular link with no singular crossings. Hence, singular links generalize classical links in $S^3$. \smallbreak

The \textit{Artin braid group}, denoted by $B_n$, is the group defined by generators $\sigma_1,\ldots,\sigma_{n-1}$ and relations:
\begin{align}
    \sigma_i \sigma_j & = \sigma_j \sigma_i \quad \textrm{ if } |i-j|\geq 2, \label{B1} \\
    \sigma_i \sigma_j \sigma_i &= \sigma_j \sigma_i \sigma_j \quad \textrm{ if } |i-j|=1, \label{B2}
\end{align}

\begin{figure}[h]
    \begin{minipage}[t]{.45\textwidth}
        \centering
         \includegraphics[scale=1]{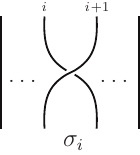}
    \end{minipage}
    \hfill
    \begin{minipage}[t]{.45\textwidth}
        \centering
        \includegraphics[scale=1]{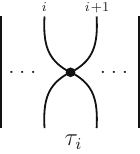}
    \end{minipage}
    \caption{Generators of $SB_n$.} \label{fig:gens SBn}
\end{figure}
\noindent where the generator $\sigma_i$ represents the braid whose $i$--th strand goes under the $(i+1)$--th strand and the rest are vertical (see Figure \ref{fig:gens SBn}). Let $\alpha$ be a braid in the group $B_n$; the closure of $\alpha$, denoted by $\widehat{\alpha}$, is obtained by joining its corresponding endpoints with simple arcs. From every closed braid $\widehat{\alpha}$ we obtain a classical link in $S^3$. Moreover, by the Alexander theorem, every link can be obtained by closing a braid. In this way, the isotopy between two equivalent links is reflected on the corresponding braids, and that is the Markov  theorem, i.e., two links as closed braids $\widehat{\gamma}_1$ and $\widehat{\gamma}_2$ are equivalent, if and only if, the braid $\gamma_1$ can be transformed into the braid $\gamma_2$ through the following moves:

\begin{itemize}
    \item[(M1)] $\alpha\beta \rightleftharpoons \beta\alpha \,$, \label{M1}
    \item[(M2)] $\beta \rightleftharpoons \beta\sigma_n^{\pm 1}\,$, \label{M2}
\end{itemize}
\noindent for any $\alpha, \beta \in B_n$. Observe that in the second move (M2), the braid $\beta$ is thought as a braid on $n+1$ strands by adding an additional strand to the right, but with do not make any distinction on the notation because there is a natural inclusion $B_n\hookrightarrow B_{n+1}$. Also, we can define the inductive limit $B_{\infty}:=\bigcup_{n\geq 1} B_n$. 

A \textit{singular braid on n strands} is the image of a smooth immersion of $n$ arcs in $S ^3$, that has finitely many double points (singular crossings). Thus, a singular braid is alike a classical braid, but with a finite number of singular crossings. The set of singular braids on $n$ strands, denoted by $SB_n$, forms a monoid with the usual concatenation, which is called the \textit{singular braid monoid} \cite{Baez92,bir93}. This monoid is presented with generators $\sigma_1, \ldots, \sigma_{n-1}, \sigma_1^{-1}, \ldots, \sigma_{n-1}^{-1}$  satisfying the usual braid relations (\ref{B1})-(\ref{B2}) and the generators $\tau_1,\ldots,\tau_{n-1}$ subject to the relations:
\begin{align}
\sigma_i\tau_j&=  \tau_i\sigma_j \quad  \text{for $|i-j|\not=1$,}\label{SB1} \\
  \sigma_i\sigma_j\tau_i&= \tau_j\sigma_i\sigma_j   \quad  \text{for $|i-j|\geq 2$,}\label{SB2}\\
     \tau_i\tau_j&=\tau_j\tau_i   \quad  \text{for $|i-j|\geq 2$.}\label{SB3}
\end{align}

\noindent Analogously, there is a natural inclusion $SB_n\hookrightarrow SB_{n+1}$ and the inductive limit $SB_{\infty}:=\bigcup_{n\geq 1} SB_n$ is well defined. 
\begin{remark}

There is a natural projection from $B_n$ onto $\mathfrak{S}_n$ induced by the mapping $\sigma_i\mapsto s_i$, where $s_i$ is the transposition $(i\ i+1)$. This projection extends to $SB_n$ by mapping $\tau_i\mapsto \sigma_i$. 
Observe that since $\SM{n}$ acts on the partition monoid $\PP{n}$, then we have natural actions of $B_n$ (and $SB_n$ respectively) on $\PP{n}$. 

\end{remark}

For singular braids and singular links there are analogues of the Alexander \cite[Lemma 2]{bir93} and Markov theorems \cite{ge2005}. Thus, we have that isotopy classes of singular links are in bijection with equivalence classes of $SB_{\infty}$. The equivalence relation $\sim$ on $SB_{\infty}$ is given by the corresponding Markov moves, two braids on $SB_{\infty}$ are equivalent if one can be obtained from the other through a finite sequence of the following moves:

\begin{enumerate}[(MS1)]
    \item $\alpha\beta \rightleftharpoons \beta\alpha$, \label{MS1}
    \item $\beta \rightleftharpoons \beta\sigma_n^{\pm 1}$, \label{MS2}
\end{enumerate}
where $\alpha,\beta \in SB_n$.

\subsection{Tied singular links and the monoid of tied singular braids} \label{ctlinkssect}
The concept of tied link was originally introduced in \cite{aiju2}0. Subsequently, the concept of combinatoric tied link, or ct--link for shorten, is introduced by the same authors in \cite{aiju2018}. Both concepts are equivalent. Therefore, from now on we will refer as ct--links for our purposes.\smallbreak
 Let $\mathcal{L}_k$ be the set of links in $S^3$ with $k$ components. Thus, a ct--link of $k$ components is an element in  $\mathcal{L}_k\times \PP{k}$. Then, the set of all ct--links is given by
 $$\mathcal{L}^t=\bigsqcup_{k\in \mathbb{N}}\mathcal{L}_k\times \PP{k}.$$
 Observe that the numbering of the components of a link is arbitrary. Then, an isotopy between two links $L$ and $L'$ induces a bijection $w_{L,L'}\in S_k$ between the components of $L$ and $L'$. We say that two ct--links $(L,I)$ and $(L',I')$ are $t$--isotopic if $L$ and $L'$ are isotopic as links and $I=w_{L,L'}(I')$.

     \begin{figure}[h!]
         \centering
         \includegraphics{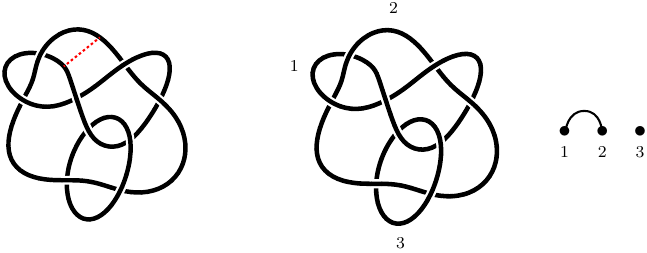}
         \caption{Equivalence between tied links and ct--links.}
         \label{tiedctlink}
     \end{figure}

\begin{remark}
 Let be $(L,I)$ a ct--link, then it can be regarded as the tied link $L$ with ties connecting the components that belong to the same block of the partition $I$, see Figure \ref{tiedctlink}.
\end{remark}

 \smallbreak
 The algebraic counterpart of ct--links (or tied links) is the  \textit{tied braid monoid}, denoted by $TB_n$,  which can be 
presented by generators $\sigma_1,\ldots ,\sigma_{n-1}$, $e_1, \ldots ,e_{n-1}$  satisfying the relations (\ref{B1})--(\ref{B2}) together with  the following relations:
\begin{align}\label{TS1}
e_ie_j & = e_j e_i,\quad 
e_i^2  = e_i, \quad e_i \sigma_i  =    \sigma_i e_i \qquad \text{for all $i,j$},
\\
\label{TS2}
e_i \sigma_j  & = 
 \sigma_je_i\qquad \text{if $\vert i  -  j\vert >1$},
\\
\label{TS3}
e_i\sigma_j\sigma_i & =   \sigma_j\sigma_ie_{j},  \quad 
e_ie_j\sigma_i  =  e_j \sigma_i e_j =  \sigma_ie_ie_j \qquad \text{if $\vert i  -  j\vert =1$},
\\
\label{TS4}
 e_i\sigma_j\sigma_i^{-1} & =  \sigma_j\sigma_i^{-1}e_j \qquad \text{if $\vert i - j\vert  = 1$}.
\end{align}

\begin{remark}
The monoid $TB_n$ can be also regarded as the semidirect product $\PP{n}\rtimes B_n$, yielded by the natural action of $B_n$ on $\PP{n}$, see \cite[Theorem 3]{aiju2018}.
\end{remark}
Furthermore, there exist analogues of the Alexander and Markov theorems in the context of tied links, for more details see \cite[Theorem 3.5]{aiju2} and \cite[Theorem 3.7]{aiju2}, respectively.\smallbreak



There also exists an analogue concept in the context of singular links. Namely, a \textit{cts-link}, or \textit{combinatoric tied singular link}, with $k$ components is a pair $(L,I)$, where $L$ is a singular link in $S^3$ and $I\in \PP{k}$.  Let $\mathcal{L}_{k,m}$ be the set of singular links in $S^3$ with $k$ components and $m$ singular crossings. Then, the set of all cts--links is given by
 $$\mathcal{L}^{ts}=\bigsqcup_{k\geq 0, m\geq 0}\mathcal{L}_{k,m}\times \PP{k}.$$
Analogously to the classical case, two cts--links $(L,I)$ and $(L',I')$ are $ts$--isotopic if $L$ and $L'$ are isotopic as singular links and $I=w_{L,L'}(I')$, where $w_{L,L'}$ is the permutation (on the set of components) induced by the isotopy between $L$ and $L'$.\smallbreak

\begin{remark} \rm
Note that a classical link $L\in \mathcal{L}_{k}$ can be considered as the ct--link $(L,\mathbf{1}_k)$. Analogously, a singular link $L'\in \mathcal{L}_{k,m}$ can be regarded as the cts--link $(L',\mathbf{1}_k)$. Since $\mathcal{L}_k=\mathcal{L}_{k,0}$, cts--links generalize ct--links, and both concepts generalize classical links.
\end{remark}

The analogue of the monoid $TB_n$ in this context is the \textit{tied singular braid monoid}, denoted by $TSB_n$. The monoid $TSB_n$ is presented by generators $\sigma_1,\ldots , \sigma_{n-1}$, $\tau_1, \ldots , \tau_{n-1}, e_1,\ldots ,e_{n-1}$   
satisfying the relations (\ref{B1})--(\ref{B2}), (\ref{SB1})--(\ref{SB3}), (\ref{TS1})-(\ref{TS4}) together with:
\begin{align}
 \tau_ie_j &= e_j\tau_i\quad\text{ if $\vert i-j\vert \not=1$,} 
 \label{TSB1}
 \\
 e_i\tau_j\tau_i & = \tau_j\tau_ie_j, \quad
e_i\tau_j\sigma_i  = \tau_j\sigma_ie_j,  \quad
e_i\sigma_j\tau_i = \sigma_j\tau_i e_j \quad\text{ if $\vert i-j\vert=1$,}
 \label{TSB2}
\\
 e_ie_j\tau_i &= e_j\tau_ie_j = \tau_ie_ie_j,
 \quad \tau_ie_j  =\sigma_ie_j\sigma_i^{-1}\tau_i \quad\text{ if $\vert i-j\vert=1$}.
  \label{TSB3}
\end{align}

This monoid can be also realized as the semidirect product $\PP{n}\rtimes SB_n$ yielded by the natural action of $SB_n$ on $\PP{n}$, see \cite[Theorem 9]{aiju2018}. Then, every element of $TSB_n$ can be regarded as pair $(I,\alpha)$, where $I\in \PP{n}$ and $\alpha\in SB_n$.\smallbreak

Additionally, there exist analogues of the Alexander and Markov theorems for cts--links, for more details see \cite[Theorem 10]{aiju2018} and \cite[Theorem 11]{aiju2018}, respectively.


\subsection{The two parameter bt--algebra}\label{bttwo}
Let $u,\sqrt{u}$ be variables such that $(\sqrt{u})^2=u$. The one paratemer bt--algebra $\mathcal{E}_n:=\mathcal{E}_n(u)$ was introduced in \cite{aiju}, and it is defined as the quotient of the monoid algebra $\mathbb{C}(\sqrt{u})[TB_n]$  over the ideal generated by

    $$\sigma_{i}^2-1 - (u - 1)e_i (1 + \sigma_i ) \quad \text{for all $i$.}$$
    
Thus, the algebra $\mathcal{E}_n(u)$ is the algebra generated by $T_1,\dots,T_{n-1}$ and $E_{1},\dots,E_{n-1}$ subject to relations (\ref{B1})--(\ref{B2}) and (\ref{TS1})--(\ref{TS4}), replacing $\sigma_i$ and $e_i$ by $T_i$ and $E_i$, respectively, together with the quadratic relations

    $$T_{i}^2=1 +(u - 1)E_i + (u - 1)E_iT_i \quad \text{for all $i$.}$$
    
Consider now the following elements in $\E_n(u)$: $$V_i=T_i+\left(\frac{1}{\sqrt{u}-1}\right)E_iT_i \quad \text{for all $i$}.$$
Then, the elements $V_i$'s also satisfy the defining relations of $\mathcal{E}_n$ replacing $T_i$ by $V_i$, however these elements satisfy the following quadratic relation
\begin{equation}\label{newquad}
  V_i^2=1+\left(\sqrt{u}-\frac{1}{\sqrt{u}}\right)E_iV_i.
\end{equation}
Thus, the bt--algebra can be also presented by generators $V_1,\ldots,V_{n-1}$ and $E_{1},\dots,E_{n-1}$ subject to the relations noted previously. We will denote the algebra with this presentation by $\E_n(\sqrt{u})$ as in \cite{aiju2018}.\smallbreak
Let $v$ be a variable commuting with $u$, and we set  $\mathbb{K}:=\mathbb{C}(u,v)$.

\begin{definition}\label{def:E(u,v)}
    The \textit{two parameter bt--algebra}, denoted by $\mathcal{E}(u,v)$ is defined as the $\mathbb{K}$--algebra presented by braid and tied generators $R_1,\dots,R_{n-1}$ and $E_1,\dots, E_{n-1}$, respectivley,  satisfying the following relations:
    \begin{align}
        E_iE_j&=E_jE_i \quad \text{for $|i-j|\geq 2$,} \label{bt1}\\
        E_i^2&=E_i \quad \text{for all $i$,}\label{bt2}\\
        E_i R_j  &= R_jE_i\quad \text{if $\vert i  -  j\vert \neq 1$},\label{bt3}\\
        E_iR_jR_i & =   R_jR_iE_j  \qquad \text{if $\vert i  -  j\vert =1$},\label{bt5}\\
        E_iE_jR_i  &=  E_j R_i E_j =  R_iE_iE_j \qquad \text{if $\vert i  -  j\vert =1$},\label{bt6}\\
R_iR_j&=R_jR_i \quad \text{for $|i-j|\geq 2$,}\label{bt7}\\
    R_jR_iR_j&=R_iR_jR_i\qquad \text{if $\vert i  -  j\vert =1$},\label{bt8}\\
    R_{i}^2&=1 +(u - 1)E_i + (v - 1)E_iR_i \quad \text{for all $i$}.\label{bt9}
 \end{align}
\end{definition}
Thus, the algebra $\EE{u}{v}$ can be regarded as the quotient of the monoid algebra $\KK(TB_n)$ over the ideal generated by 
$\sigma_{i}^2 -1 -(u - 1)e_i - (v - 1)e_i\sigma_i $.
Note also that $R_i$ is invertible. Namely, we have
\begin{equation}\label{invbt}
    R_{i}^{-1}=R_i +( 1-v)u^{-1}E_i + (u^{-1} - 1)E_iR_i.
\end{equation}

\begin{remark}
Note that the algebra $\E_n(u)$ corresponds to $\EE{u}{u}$ and $\E_n(\sqrt{u})$ is obtained as $\EE{1}{v}$, with $v=\sqrt{u}-\sqrt{u}^{-1}+1$.
\end{remark}

The algebra $\EE{u}{v}$  is isomorphic to a one parameter bt--algebra, see \cite[Proposition 1]{aiju2020}. Then, this algebra supports a Markov trace since the original bt--algebra does \cite[Theorem 3]{aijuMMJ}. More precesily, let $\A$ and $\bb$ variables that commute each other and with $u$ and $v$. There exists a unique Markov trace $\Tr{}:=\{\Tr{n}:\EE{u}{v}\to \KK(\A,\bb) \}_{n\geq 1}$ on $\EE{u}{v}$, where $\Tr{n}$'s are linear maps satisfying $\Tr{n}(1)=1$ and the following rules:
\begin{enumerate}[(i)]
    \item $\Tr{n}(XY)=\Tr{n}(YX),$
    \item $\Tr{n+1}(XR_n)=\Tr{n+1}(XE_nR_n)=\A\Tr{n}(X),$
    \item $\Tr{n+1}(XE_n)=\bb\Tr{n}(X),$
\end{enumerate}
\noindent where $X,Y\in \EE{u}{v}$, see \cite[Propositon 2]{aiju2020}. 

\begin{remark}\label{rem:compinv}
   	Let be $\Delta$, $\Theta$ and $\Upsilon$ are the invariants of classical links that are derived by using Jones method on algebras $\mathcal{E}(u)$, $\mathcal{E}(\sqrt{u})$ and $\EE{u}{v}$, respectively. Recall that the invariants  $\Delta$ and $\Theta$ correspond
to specializations of $\Upsilon$. More precisely, we can recover them by setting $u = v$ and $u = 1$ with $v = \sqrt{u} - \sqrt{u}^{-1}+1$, respectively.\smallbreak  In \cite{chjukala}, it is proved that $\Theta$ coincides with the HOMFLY--PT polynomial on knots, but it distinguishes links that are not distinguished by the HOMFLY--PT polynomial. Similarly, the invariant $\Delta$ is also more powerful than the HOMFLY--PT polynomial. However, the pairs distinguished by $\Delta$ and $\Theta$ do not coincide in some cases. Additionally, we have that $\Upsilon$ distinguishes pairs that are not distinguished by $\Delta$ or $\Theta$. For instance, we have the following table 
	
\begin{center}
	\begin{tabular}{|c|c|c|c|c|}
		\hline
		Link & Link & $\Upsilon$& $\Theta$& $\Delta$ \\
		\hline
		$L11n358\{0, 1\}$ &  $L11n418\{0, 0\}$ & $\bullet$ & $ \bullet$ &  \\
		\hline
		$L11n358\{1, 1\}$& $L11n418\{1, 0\}$ & $\bullet$  &  & $\bullet$  \\
		\hline
		$L11n356\{1, 0\}$ &  $L11n434\{0, 0\}$ & $\bullet $  &  & $\bullet $ \\
		\hline
		$L11a467\{0, 1\}$ &  $L11a527\{0, 0\}$ &$\bullet$  & $\bullet$  &  \\
		\hline
	\end{tabular}
\end{center}

\noindent where the dot means that the pairs of links in the left are distinguished by the invariant in the corresponding column, for details see \cite[Section 6]{aiju2020}. The notation used here is taken from the web page LinkInfo \cite{linkinfo}.
\end{remark}

\subsection{Some known invariants for singular links}
In this section, we recall two invariants for singular links in literature that are derived via Jones method. This invariants motivate our approach in the sequel.
\subsubsection{The Paris--Rabenda invariant}
Let  $H_n:=H_n(u)$ be the Iwahori--Hecke algebra, that is the $\mathbb{C}(u)$--algebra presented by generators $h_1,\ldots, h_{n-1}$ subject to the following relations:
\begin{align}
    h_i h_j &= h_j h_i \qquad \text{if $|i-j|>1$,} \label{H1}\\
    h_i h_j h_i &= h_j h_i h_j \qquad \text{if $|i-j|=1$,} \label{H2}\\
    h_i^2 &= u+(u-1) h_i \qquad \text{for all $i$.} \label{H3}
\end{align}
Thus, $H_n$ is the quotient of the group algebra $\mathbb{C}(u)[B_n]$ over the ideal generated by the quadratic relations (\ref{H3}).


In \cite{pare}, Paris and Rabenda introduce the \textit{Singular Hecke algebra}, denoted by $SH_n:=SH_n(u)$. This algebra is a natural generalization of the Iwahori--Hecke algebra and it is defined as the quotient of the monoid algebra $\mathbb{C}(u)[SB_n]$ by the ideal generated by the elements $\sigma_i^2-u-(u-1)\sigma_i, $ for $0\leq i\leq n-1$. Let be $g_i$ and $p_i$ the class of $\sigma_i$ and $\tau_i$ in the quotient, respectively. Thus, $SH_n$ is presented by braid generators $g_1,\ldots,
g_{n-1}$ and singular generators $p_1,\dots,p_{n-1}$ subject to the following relation
\begin{align}
g_ig_j  =  g_j g_i, &\quad p_ip_j= p_j p_i,\quad p_ig_j= g_j p_i \qquad \text{if $\vert i - j\vert \not= 1$},
 \label{SH1}\\
g_ig_jg_i = g_jg_ig_j, &\quad  g_ig_jp_i = p_jg_ig_j \qquad \text{if $\vert i - j\vert = 1$}, \label{SH2}\\
g_i^2 = & \,u + (u-1)g_i \qquad \text{for all $i$.}\label{SH3}
\end{align}
For $d\geq 0$, let $S_dB_n$ be the set of singular braids with $d$ singular points. Then, we set $S_dH_n$ as the $\mathbb{C}(u)$--linear subspace of $SH_n$ generated by $S_dB_n$. Thus, we have the following graduation on the algebra
$$SH_n=\bigoplus_{d=0}^{\infty} S_dH_n.$$
Observe that $S_0H_n$ is the Iwahori--Hecke algebra $H_n$.\smallbreak



    Let $\A$ be  a variable and $d\geq 0$. A \textit{d-Markov trace} on $\{S_dH_n\}_{n=1}^{\infty}$ is a family of $\mathbb{C}(u)$--linear maps $\tr^d:=\{\tr_n^d\}_{n=0}^{\infty}$, where $\tr_n^d:S_dH_n\rightarrow \mathbb{C}(u,\A)$, satisfying:
    \begin{enumerate}[(i)]
        \item $\tr_n^d(\alpha\beta)=\tr_n^d(\beta\alpha)$ for all $\alpha \in S_kB_n$ and $\beta \in S_{d-k}B_n$,
        \item $tr^d_{n+1}(\alpha)=\tr_n^d(\alpha)$ for all $\alpha \in S_dB_n$,
        \item $\tr_{n+1}^d(\alpha g_n)=\A\tr_n^d(\alpha)$ for all $\alpha \in S_dB_n$,
    \end{enumerate}
    for $n\geq 1$. Thus, a \textit{Markov trace} on $\{SH_n\}_{n=1}^{\infty}$ is a family of $\mathbb{C}(u)$--linear maps $\tr:=\{\tr_n\}_{n=0}^{\infty}$, such that $\{\tr_n^d:=\tr_n|_{S_dH_n}\}_{n=1}^{\infty}$ is a $d$-Markov trace for all $d\geq 0$. 

For $d\geq 0$, the family of $d$--Markov traces $\{\tr^{(d,0)},\tr^{(d,1)},\ldots,\tr^{(d,d)}\}$ can be constructed, for details see  \cite[Section 4]{pare}. This family is a basis for $\T_d$, the $\mathbb{C}(u)$--vector space of all $d$-Markov traces on $\{S_dH_n\}_{n=1}^{\infty}$, see \cite[Theorem 4.7]{pare}. Let $X,Y$ new variables that commute each other, the so called universal trace $\widehat{\tr}:=\{\widehat{\tr}_n\}_{n=0}^{\infty}$ is defined as follows: 
$$\widehat{\tr}_n(\alpha):=\sum_{k=0}^{d}\frac{\sqrt{u}^k}{(d-k)!k!}X^k Y^{d-k}\tr_n^{(d,k)}(\alpha),$$
for $d\geq 0$ and $\alpha\in S_dH_n$. 

Thus, the Paris--Rabenda invariant $\widehat{I}$ is derived by using Jones method. More precisely, let $\pi:SB_n\to SH_n$ be the natural representation, and $\varepsilon:SB_n\to \mathbb{Z}$ the monoid homomorphism given by $\sigma_i^{\pm 1}\mapsto \pm 1$ and $\tau_i\mapsto 0$. Then, for a singular link $L$, the invariant $\widehat{I}$ is given by
\begin{equation}\label{invpr}
    \widehat{I}(L):=\left(\frac{1-u\lambda}{u-1} \right)^{n-1}(\sqrt{\lambda})^{\varepsilon(\alpha)-n+1}\widehat{\tr}_n(\pi(\alpha)),
\end{equation}
where $\alpha \in SB_n$ such that $\widehat{\alpha}=L$, and $\lambda:=\frac{\A-u+1}{u\A}$.

\begin{remark}
  It is well known, that singular links can be also regarded as a four-valent graph with rigid vertices embedded in $\RR^3$. In \cite{KauVog92}, Kauffman and Vogel give a method to extend any regular isotopy invariant of unoriented (oriented) classical links  to a regular isotopy invariant of unoriented (oriented) rigid vertex graphs in $\RR^3$ by imposing certain skein rule involving the rigid vertex, for details see \cite[Section 1]{KauVog92}. They apply this method to an state model for the HOMFLY--PT polynomial obtaining an invariant of four-valent graph that satisfies the HOMFLY--PT skein relation togheter with the skein rule imposed to the rigid vertex. Thus, the Kauffman Vogel invariant can be obtained as a specialization of $\widehat{I}$, for details see \cite[Lemma 5.6]{pare}.
\end{remark}

\subsubsection{The Aicardi--Juyumaya invariants}
In \cite{aiju2020}, four invariants for singular links are defined. All of them are derived by using Jones method on the Markov trace on the bt--algebra $\E_n(u)$, though using different homomorphism from the monoid $SB_n$ to $\E_n(u)$ in each case. More precisely, let $x,y$ be commuting variables, we then have that the mapping $\sigma_i\mapsto wT_i$ and $\tau_i\mapsto x+wyT_i$ (resp. $\sigma_i\mapsto wT_i$ ; $\tau_i\mapsto E_ix+wyE_iT_i$) induces an homormophism $\psi_{w,x,y}:SB_n\rightarrow\E_n(u)$ (resp. $\phi_{w,x,y}:SB_n\rightarrow\E_n(u)$). Additionally, the mappings obtained by replacing $T_i$ by $V_i$ also induce two homomorphisms from $SB_n$ to $\E_n(\sqrt{u})$, denoted by $\psi'_{w,x,y}$ and $\phi'_{w,x,y}$, respectively. See \cite[Proposition 6]{aiju2020}.\smallbreak
For any singular link $L$, the following functions are defined
\begin{align}
    \Psi_{x,y}(L)&:=\left(\frac{1}{\A \sqrt{\mathsf{c}}}\right)^{n-1}(\Tr{n}\circ\psi_{w,x,y})(\alpha) \in \KK(\bb,w,x,y),\label{invai1}\\
    \Phi_{x,y}(L)&:=\left(\frac{1}{\A \sqrt{\mathsf{c}}}\right)^{n-1}(\Tr{n}\circ\phi_{w,x,y})(\alpha) \in \KK(\bb,w,x,y),\label{invai2}\\
    \Psi'_{x,y}(L)&:=\left(\frac{1}{\A \sqrt{\mathsf{d}}}\right)^{n-1}(\Tr{n}\circ\psi'_{w,x,y})(\alpha) \in \KK(\bb,w,x,y),\label{invai1'}\\
    \Phi'_{x,y}(L)&:=\left(\frac{1}{\A \sqrt{\mathsf{d}}}\right)^{n-1}(\Tr{n}\circ\phi'_{w,x,y})(\alpha) \in \KK(\bb,w,x,y),\label{invai2'}
\end{align}

  
  \noindent where $\alpha \in SB_n$ and $L=\widehat{\alpha}$. Then, the functions $\Psi_{x,y}$ and $\Phi_{x,y}$ (resp. $\Psi'_{x,y}$ and $\Phi'_{x,y}$) are singular links invariant by setting 
$$ \mathsf{c}:=\frac{\A+(1-u)\bb}{\A u}\quad  \left(\text{resp.}\ \mathsf{d}:=\frac{\A+(1-u)\bb}{\A }\right).$$
  For details see \cite[Theorem 7]{aiju2020} and \cite[Theorem 8]{aiju2020}, respectively. Moreover, these invariants can be extended to invariants of $cts$--links, which are denoted again by $\Psi_{x,y}, \Phi_{x,y}, \Psi'_{x,y}$ and $\Phi'_{x,y}$.
  
    \begin{remark}
      Observe that for a classical link $L$ we have that 
      $\Psi_{x,y}=\Phi_{x,y}=\Delta$ and $ \Psi'_{x,y}=\Phi'_{x,y}=\Theta,$ where $\Delta$ and $\Theta$ are the classical links invariants derived from the algebras $\mathcal{E}_n(u)$ and $\mathcal{E}_n(\sqrt{u})$, respectively.
    \end{remark}
  
  \begin{remark}
  There exist pairs of singular links that are distinguished by $\Psi_{x,y}, \Phi_{x,y}, \Psi'_{x,y}$ and $\Phi'_{x,y}$  but not distinguished by $\widehat{I}$. Likewise, there exist pairs of singular links that are distinguished by $\Psi_{x,y}$ and $\Phi_{x,y}$ but not by $\Psi'_{x,y}$ and $\Phi'_{x,y}$. Finally, we also have that $\Psi_{x,y}$ is more powerful than $\Phi_{x,y}$ on cts--links. For details see \cite[Section 8]{aiju2020}.
  \end{remark}

\section{New invariants for singular links and cts--links}\label{newinv}
In this section, we introduce new invariants for singular links and cts--links following the approach in \cite{aiju2020}, but this time using the two parameter bt--algebra. Thus, the four invariants $\Psi_{x,y}, \Phi_{x.y}, \Psi'_{xy}$ and $\Phi'_{x,y}$ given in \cite[Section 5]{aiju2018} can be recovered by specializing the invariants defined in this section.\smallbreak
Let $w,x,y$ and $z$ be variables that commute each other and with $\A$ and $\bb$.  We set $\mathbb{L}:=\KK(w,x,y,z)$; from now on we work in the algebra $\EE{u}{v}\otimes_{\KK} \mathbb{L}$, which will be denoted again by $\EE{u}{v}$ for simplicity. We also denote by $\Tr{}$ the Markov trace defined in Section \ref{bttwo} extended to the $\LL$-linear map from  $\EE{u}{v}\otimes_{\KK} \mathbb{L}$ into $\LL(\A,\bb)$.
\begin{proposition}\label{newrep}
\begin{enumerate}[(i)]
The following statements hold 
    \item The mappings $\sigma_i\mapsto wR_i$ and $\tau_i\mapsto x+ywR_i+zw^{-1}R_i^{-1}$ induce a monoid homomorphism  $\rho_{w,x,y,z}:SB_n\rightarrow \E_n(u,v)$.
     \item The mappings $\sigma_i\mapsto wR_i$ and $\tau_i\mapsto 
     E_i(x+ywR_i)$ induce a monoid homomorphism  $\varrho_{w,x,y}:SB_n\rightarrow \E_n(u,v)$.
\end{enumerate}
\end{proposition}

   \begin{proof}
  The proof reduces to verify that the images of $\sigma_i$ and $\tau_i$ (in each case) satisfy the relations of $SB_n$. It is not difficult to check that (\ref{B1})--(\ref{SB1}) and (\ref{SB3}) are satisfied in both cases. For (\ref{SB2}), we have
  \begin{equation}\label{rephold}
      \begin{array}{rcl}
      w^2R_iR_j(x+ywR_i+zw^{-1}R_i^{-1}) &=&  w^2xR_iR_j+w^3yR_iR_jR_i+wzR_iR_jR_i^{-1}\\
      &\stackrel{(\ref{bt8})}{=}&w^2xR_iR_j+w^3yR_jR_iR_j+wzR_j^{-1}R_iR_j\\
      &=&(x+ywR_j+zw^{-1}R_j^{-1})w^2R_iR_j,
  \end{array}
\end{equation}
  \noindent and therefore (i) holds. Finally, (ii) follows by Eq. (\ref{rephold}) and (\ref{bt5}) and setting $z=0$.
   \end{proof}
   
From now on, to simplify notation we will denote the representations $\rho_{w,x,y,z}$ and $\varrho_{w,x,y}$ by $\rho$ and $\varrho$, respectively.
We use these homomorphism to derive invariants for singular links. Up to normalization, we want that $\Tr{n}\circ\rho$ and $\Tr{n}\circ\varrho$ define singular link invariants,  hence we need that they respect the Markov move (MS\ref{MS2}). Then, we set
    \begin{equation}\label{rescale}
      w^2=\frac{\A+(1-v)\bb}{\A u}.
    \end{equation}
Thus, for a singular link $L$ we define:
    \begin{equation}\label{invsin1}
        \Upsilon_{x,y,z}(L):=\left(\frac{1}{\A w}\right)^{n-1}(\Tr{n}\circ \rho)(\alpha) \in \LL(\A,\bb)
    \end{equation}
and
    \begin{equation}\label{invsin2}
        \Upsilon'_{x,y}(L):=\left(\frac{1}{\A w}\right)^{n-1}(\Tr{n}\circ\varrho)(\alpha) \in \LL(\A,\bb),
    \end{equation}
where $\alpha \in SB_n$ and $L=\widehat{\alpha}$.

\begin{remark}
    Observe that Eq.(\ref{invbt}) implies that the addition of an extra  parameter $z$ in the definition of the representation $\varrho$ is superflous for the new invariant $\Upsilon'_{x,y}$. Namely, we have that 
    $$E_iR_{i}^{-1}=E_iR_i +( 1-v)u^{-1}E_i + (u^{-1} - 1)E_iR_i=( 1-v)u^{-1}E_i + u^{-1}E_iR_i.$$
    However, using an extra parameter in representation $\rho$ should be meaningful for the invariant $\Upsilon_{x,y,z}$, since $R_i^{-1}$ can't be expressed as a linear combination of $R_i$ and $1$.
\end{remark}

\begin{theorem}\label{newinvariants}
The functions  $\Upsilon_{x,y,z}$ and $\Upsilon'_{x,y}$ are isotopy invariants of singular links.
\end{theorem}
  \begin{proof}
  We have to verify that the functions  $\Upsilon_{x,y,z}$ and  $\Upsilon'_{x,y}$ respect the Markov moves (MS\ref{MS1}) and (MS\ref{MS2}). Firstly, note that both functions respect move (MS\ref{MS1}) by property (i) of the trace $\Tr{}$. We only verify that  $\Upsilon_{x,y,z}$ respects the move (MS\ref{MS2}) since the other case follows analogously. Let be $\alpha \in SB_n$, we then have
 \begin{align*}
    (\Tr{n+1}\circ\rho)(\alpha \sigma_n^{-1}) &= w^{-1}\Tr{n+1}(\rho(\alpha) R_n^{-1}) \\
     &=w^{-1} \Tr{n+1}(\rho(\alpha)(R_n +( 1-v)u^{-1}E_n + (u^{-1} - 1)E_nR_n)) \\
     &=(\Tr{n}\circ\rho)(\alpha)w^{-1}(\A +( 1-v)u^{-1}\bb + (u^{-1} - 1)\A) \\
     &=(\Tr{n}\circ\rho)(\alpha)w^{-1} \left(\frac{\A +( 1-v)\bb}{u} \right) \stackrel{(\ref{rescale})}{=} \A w(\Tr{n}\circ\rho)(\alpha).
\end{align*}

  Thus, we obtain 
  $$\Upsilon_{x,y,z}(\widehat{\alpha\sigma_n^{-1}})=\left(\frac{1}{\A w}\right)^{n}\A w(\Tr{n}\circ \rho)(\alpha)=\Upsilon_{x,y,z}(\widehat{\alpha}).$$
    We can prove $\Upsilon_{x,y,z}(\widehat{\alpha\sigma_n})=\Upsilon_{x,y,z}(\widehat{\alpha})$ in a similar way.
  \end{proof}

\begin{remark}
    For $L$ a classical link, we have that 
     $$\Upsilon_{x,y,z}(L)=\Upsilon'_{x,y}(L)=\Upsilon(L),$$
     where $\Upsilon$ is the classical link invariant derived from the two parameter algebra and its Markov trace, see Remark \ref{rem:compinv}. Thus, by setting $z=0$ and $v=u$ (resp. $u=1$ and $v=\sqrt{u}-\sqrt{u}^{-1}+1$), we recover the invariant $\Psi_{x,y}$ (resp. $\Psi_{x,y}'$ ) defined on \cite[Section 5]{aiju2018}.
\end{remark}

\begin{proposition}{\label{prop:ext}}
     The representations $\rho$ and $\varrho$ can be naturally extended to the tied singular braid monoid $TSB_n$ by mapping $e_i\mapsto E_i$. This extensions will be denoted by $\widetilde{\rho}$ and $\widetilde{\varrho}$, respectively.
\end{proposition}

\begin{proof}
    It follows by Proposition \ref{newrep} and the defining relations of the algebra $\EE{u}{v}$
\end{proof}
Thus, for $(L,I)$ a cts--link, we can define the functions

\begin{equation}\label{eq:invcts1}
      \widetilde{\Upsilon}_{x,y,z}(L,I):=\left(\frac{1}{\A w}\right)^{n-1}(\Tr{n}\circ\widetilde{\rho})(\alpha) \in \LL(\bb)
  \end{equation}

 and  

 \begin{equation}\label{eq:invcts2}
      \widetilde{\Upsilon'}_{x,y}(L,I):=\left(\frac{1}{\A w}\right)^{n-1}(\Tr{n}\circ\widetilde{\varrho})(\alpha) \in \LL(\bb).
  \end{equation}

\begin{proposition}
 The functions  $\widetilde{\Upsilon}_{x,y,z}$ and $\widetilde{\Upsilon'}_{x,y}$ are isotopy invariants of cts-links.   
\end{proposition}
\begin{proof}
    The result follows by Theorem \ref{newinvariants}.
    \end{proof}
\begin{remark}\label{singversion}
For a $k$-component singular link $L$ we have that 
    $$\widetilde{\Upsilon}_{x,y,z}(L,\mathbf{1}_{k})=\Upsilon_{x,y,z}(L)\quad \text{and} \quad \widetilde{\Upsilon'}_{x,y}(L,\mathbf{1}_{k})=\Upsilon'_{x,y}(L)$$
    where $\mathbf{1}_{k}:=(\{1\},\{2\},\dots,\{k\})$ the partition in $\PP{k}$ with only single blocks. 
\end{remark}

\begin{remark}\label{tiedtoctlink}
   In \cite[Theorem 3]{aiju2020}, the authors give skein relations for $\widetilde{\Upsilon}$, the invariant of tied links that extends $\Upsilon$, see \cite[Section 4.2]{aiju2020}. These are given using tied links notation since the concept of cts--link would be 
subsequently introduced in \cite{aiju2018}. More precisely, we know that $\widetilde{\Upsilon}$ satisfies the following skein relation
$$\frac{1}{w}\widetilde{\Upsilon}(L_+)-w\widetilde{\Upsilon}(L_{-})=\frac{v-1}{u}\widetilde{\Upsilon}(L_{\sim})+\frac{1-u^{-1}}{w}\widetilde{\Upsilon} (L_{+,\sim}),$$
where $L_+$, $L_-$, $L_{\sim}$ and $L_{+,\sim}$ are the diagrams of tied links that only differ inside a neighborhood of a crossing as in Figure \ref{fig:L+-o tied}.

\begin{figure}[h!]
    \centering
    \includegraphics[width=0.6\linewidth]{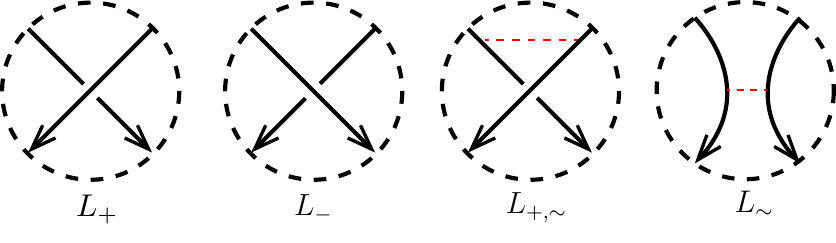}
    \caption{Diagrams of $L_{+},L_-,L_{+,\sim}$ and $L_{\sim}$.}
    \label{fig:L+-o tied}
\end{figure}

Notice that when we apply a smoothing with tie ($L_{\sim}$) to a crossing, we have two possible cases. If the crossing involves two different components, then these components merge in a single one, hence the tie added becomes superfluous. On the other hand, if the crossing involves only one component, it divides into two new components that are tied together. Therefore, in the context of cts--links, in both cases, the link, the number of components and also the partition of its components change. 
\end{remark}

Considering the previous remark, we now introduce notation in order to give skein relations for the invariant $\widetilde{\Upsilon}_{x,y,z}$ and $\widetilde{\Upsilon'}_{x,y}$. For $I\in \PP{k}$ and $1\leq i,j\leq k$.
	\begin{enumerate}[(a)]
		\item We denote by $I^+$ the element in $\PP{k+1}$ obtained from $I$ by adding the single block $\{k+1\}$. Then, $I^+=\iota(I)$, where $\iota:\PP{k} \hookrightarrow \PP{k+1}$ is the natural inclusion.
		\item We denote by $I_{i,j}$ the element in $\PP{k}$ obatined from $I$ by joining the blocks that contain $i$ and $j$, respectively (they are not necessarily disjoint).
		\item For $i<j$, we define $\tilde{I}_{i,j}$ as the element in $\PP{k-1}$ obtained from $I$ by ommiting $j$.

  \item For $i=j$, we define $\tilde{I}_{i,i}$ as the element in $\PP{k+1}$ obtained by joining $k+1$ to the block that contains $i$.
   \item We define $$I^*_{i,j}:=\left\{\begin{array}{cc}
     \tilde{I}_{i,j}  &  i<j\\
      I^+ &   \textit{i=j}.
  \end{array}\right.$$
	\end{enumerate} 

  Let be $L_{\times},L_+,L_-$ and $L_0$ singular links that only differ inside a neighborhood of a crossing as in the following Figure \ref{fig:L+-o}, and let be $C_i$ and $C_j$ the components that interact in such crossing (which are not necessary distinct).

\begin{figure}[h!]
    \centering
    \includegraphics[width=0.6\linewidth]{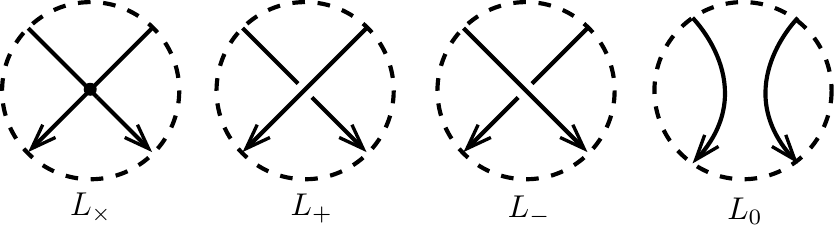}
    \caption{Diagrams of $L_{\times},L_+,L_-$ and $L_0$.}
    \label{fig:L+-o}
\end{figure}

	\begin{theorem}\label{skeinth}
	   The value of the invariant $\widetilde{\Upsilon}_{x,y,z}$ on a $k$ component cts-link $(L,I)$ is uniquely determined by the following rules:\smallbreak

		\begin{enumerate}[(1)]
			\item\label{initial1} $\widetilde{\Upsilon}_{x,y,z}(\bigcirc)=1$, where $\bigcirc$ denotes the unknotted circle.
			\item \label{initial2} $\widetilde{\Upsilon}_{x,y,z}(L\sqcup \bigcirc,I^+)=\frac{1}{aw}\widetilde{\Upsilon}_{x,y,z}(L,I).$
   
			\item\label{skeintie}Skein relation: $$\frac{1}{w}\widetilde{\Upsilon}_{x,y,z}(L_+,I)-w\widetilde{\Upsilon}_{x,y,z}(L_-,I)=\left(\frac{v-1}{u}\right)\widetilde{\Upsilon}_{x,y,z}(L_+,I_{i,j})+\frac{(1-u^{-1})}{w}\widetilde{\Upsilon}_{x,y,z}(L_0,\tilde{I}_{i,j}).$$
            \item\label{desingtie} Desingularization rule:
            $$\widetilde{\Upsilon}_{x,y,z}(L_{\times},I)=x\widetilde{\Upsilon}_{x,y,z}(L_0,I^*_{i,j})+y\widetilde{\Upsilon}_{x,y,z}(L_+,I)+z\widetilde{\Upsilon}_{x,y,z}(L_-,I).$$
		\end{enumerate}

	\end{theorem}

    \begin{proof}
        For ct--links the invariant $\widetilde{\Upsilon}_{x,y,z}$ coincides with the invariant $\widetilde{\Upsilon}$ and we know that this invariant satisfies the rules (\ref{initial1})--(\ref{skeintie}), see \cite[Theorem 3]{aiju2020}. To prove (\ref{desingtie}), suppose $L_{\times}=\widehat{\alpha\tau_i\beta}$ for some $\alpha,\beta \in SB_n$ and $1\leq i\leq n-1$. We then have that $L_0=\widehat{\alpha\tau_i\beta}$, $L_+=\widehat{\alpha\g{i}\beta}$ and $L_-=\widehat{\alpha\g{i}^{-1}\beta}$.  By definition of the homomorphism $\rho$ we have that
        \begin{align*}
            \widetilde{\rho}(\alpha\tau_i\beta)&=x\widetilde{\rho}(\alpha\beta)+y\widetilde{\rho}(\alpha) w\sigma_i\widetilde{\rho}(\beta)+z\widetilde{\rho}(\alpha)w^{-1}\sigma_i^{-1}\widetilde{\rho}(\beta)\\
            &=x \tilde{\rho}(\alpha\beta)+y\widetilde{\rho}(\alpha\sigma_i\beta)+z\widetilde{\rho}(\alpha\sigma_i^{-1}\beta).
        \end{align*}
        
        Then, the result follows by applying trace and  normalizing.
    \end{proof}

\begin{theorem}\label{skeinth2}
    The value of the invariant $\widetilde{\Upsilon'}_{x,y}$ on a $k$ component cts-link $(L,I)$ is uniquely determined by the same rules (\ref{initial1})-(\ref{skeintie}) as $\widetilde{\Upsilon}_{x,y,z}$ in Theorem \ref{skeinth} tough the desingularization is replaced by 
    \begin{equation}\label{desing2}
   \widetilde{\Upsilon'}_{x,y}(L_{\times},I)=x\widetilde{\Upsilon'}_{x,y}(L_0,\tilde{I}_{i,j})+y\widetilde{\Upsilon'}_{x,y}(L_+,I_{i,j}).
    \end{equation}
\end{theorem}  
 \begin{proof}
     It suffices to prove that $\widetilde{\Upsilon'}_{x,y}$ satisfies (\ref{desing2}). Thus, suppose $L_{\times}=\widehat{\alpha\tau_i\beta}$ for some $\alpha,\beta \in SB_n$ and $1\leq i\leq n-1$. Then $L_0=\widehat{\alpha\tau_i\beta}$ and $L_+=\widehat{\alpha\g{i}\beta}$.  By definition of the homomorphism $\widetilde{\varrho}$ we have that
\begin{align*}
    \widetilde{\varrho}(\alpha\tau_i\beta)&=\widetilde{\varrho}(\alpha) (xE_i)\widetilde{\varrho}(\beta)+y\widetilde{\varrho}_{w,x,y,z}(\alpha) (wE_i\sigma_i)\widetilde{\varrho}(\beta)\\
    &=x\widetilde{\varrho}(\alpha e_i\beta)+y\widetilde{\varrho}(\alpha\sigma_ie_i\beta).
\end{align*}

Then, the result follows by applying trace and  normalizing.
 \end{proof}
 
 \begin{corollary}\label{corsing}
  The singular link invariant $\Upsilon_{x,y,z}$ satisfies the following desingularization rule,
  \begin{equation}\label{desrule}
       \Upsilon_{x,y,z}(L_{\times})=x\Upsilon_{x,y,z}(L_0)+y\Upsilon_{x,y,z}(L_+)+z\Upsilon_{x,y,z}(L_-).
  \end{equation}

  \end{corollary}
  \begin{proof}
      It follows by Remark \ref{singversion} and Theorem \ref{skeinth}.
  \end{proof}

\begin{remark}\label{uniontied}
    Let $L$ be a singular link of $k$ components and $I\in \PP{k}$. Consider the cts-link $(L',\widetilde{I}_{j,j})$ where $1\leq j\leq k$ and $L'=L\sqcup \bigcirc$. Then, skein relation (\ref{skeintie}) from Theorem \ref{skeinth} implies that
    $$\widetilde{\Upsilon}_{x,y,z}(L', \widetilde{I}_{j,j})=\frac{\bb}{\A w}\widetilde{\Upsilon}_{x,y,z}(L, I).$$
\end{remark}
  
 \subsection{Comparison of invariants}
  In this section, we give a initial comparison between the Aicardi--Juyumaya invariants, the Paris--Rabenda invariant and the invariants $\widetilde{\Upsilon}_{x,y,z}$ and  $\widetilde{\Upsilon'}_{x,y}$ introduced above.\smallbreak
  Let be $P$ the HOMFLY--PT polynomial, $\widetilde{\Delta}$ and $\widetilde{\Theta}$ the extensions to ct--links (tied links) of $\Delta$ and $\Theta$, respectively. Recall that by \cite[Proposition 9]{aiju2018}, we have:
  \begin{enumerate}[(i)]
      \item For a classical link $L$, $$\hat{I}(L)=P(L).$$
      \item For a classical link $L$ and $I$ the partition with a unique block,
      $$\widetilde{\Delta}(L,I)=\widetilde{\Theta}(L,I)=P(L).$$
      \item For a classical link $L$ with $k$ components,
      $$\Phi_{x,y}(L,1_k)=\Psi_{x,y}(L,1_k)=\Delta(L)\quad \text{and} \quad \Phi'_{x,y}(L,1_k)=\Psi'_{x,y}(L,1_k)=\Theta(L). $$
       \item For a singular link $L$ and $I$ the partition with a unique block,
      $$\Phi_{x,y}(L,I)=\Psi_{x,y}(L,I)=\widehat{I}(L).$$
  \end{enumerate}

 \noindent Additionally, by \cite[Remark 8]{aiju2020} implies that:
 \begin{enumerate}[(i)]
\item For a classical link $L$ and $I$ the partition with a unique block,
      $$\widetilde{\Upsilon}_{x,y,z}(L,I)=\widetilde{\Upsilon'}_{x,y}(L,I)=P(L).$$
     \item For a singular link $L$ and $I$ the partition with a unique block,
      $$\widetilde{\Upsilon}_{x,y,z}(L,I)=\widetilde{\Upsilon'}_{x,y}(L,I)=\widehat{I}(L).$$
 \end{enumerate}
\begin{remark}
Note that the invariants $\widetilde{\Delta}$ and $\widetilde{\Theta}$ are equivalent to the invariants $\mathcal{F}$ and $\mathcal{F}'$, respectively, which are used in \cite[Section 8]{aiju2018}, for details see \cite[Theorem 5]{aiju2}.
\end{remark}

\begin{figure}[h!]
    \centering
    \includegraphics{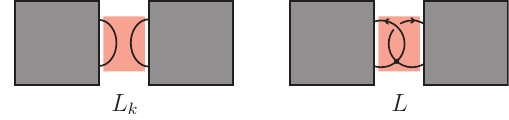}
    \caption{}
    \label{fig:L_kyL}
\end{figure}

\begin{proposition}\label{prop:compinv}
  Let $L_k$ be the union of $k$ circles unlinked. If $L$ is the link obtained from $L_k$ by modifying it locally  as in Figure \ref{fig:L_kyL}. Then,
    \begin{equation*}
        \widetilde{\Upsilon}_{x,y,z}(L,1_k)\not=\widetilde{\Upsilon}_{x,y,z}(L,\{\{i,j\}\}) \quad \text{and}\quad \widetilde{\Upsilon'}_{x,y}(L,I_k)=\widetilde{\Upsilon'}_{x,y}(L,\{\{i,j\}\}).
    \end{equation*}
\end{proposition}

\begin{proof}
 Let be $C_i$ and $C_j$ the components of $L$ that interact in its singular crossing. Then, applying desingularization rule from Theorem \ref{skeinth} we have
$$\widetilde{\Upsilon}_{x,y,z}(L,1_k)=x\widetilde{\Upsilon}_{x,y,z}(L_0,1_{k-1})+y\widetilde{\Upsilon}_{x,y,z}(L_+,1_k)+z\widetilde{\Upsilon}_{x,y,z}(L_-,1_k),$$
and 
$$\widetilde{\Upsilon}_{x,y,z}(L,\{\{i,j\}\})=x\widetilde{\Upsilon}_{x,y,z}(L_0,1_{k-1})+y\widetilde{\Upsilon}_{x,y,z}(L_+,\{\{i,j\}\})+z\widetilde{\Upsilon}_{x,y,z}(L_-,\{\{i,j\}\}).$$

   Thus, $ \widetilde{\Upsilon}_{x,y,z}(L,1_k)\not=\widetilde{\Upsilon}_{x,y,z}(L,\{\{i,j\}\})$ by Remark \ref{uniontied}. On the other hand, applying desingularization rule from Theorem \ref{skeinth2} we obtain
$$\widetilde{\Upsilon'}_{x,y}(L,1_k)=\widetilde{\Upsilon'}_{x,y}(L,\{\{i,j\}\})=x\widetilde{\Upsilon'}_{x,y}(L_0,1_{k-1})+y\widetilde{\Upsilon'}_{x,y,z}(L_+,\{\{i,j\}\}),$$
   and the proof follows.
\end{proof}

 Consequently, the invariant $\widetilde{\Upsilon}_{x,y,z}$ is stronger than the invariant $\widetilde{\Upsilon'}_{x,y}$ on cts--links, cf. \cite[Proposition 10]{aiju2018}. 
 Moreover, the following result shows that $\Upsilon_{x,y,z}$ distinguishes pairs of singular links that $\Psi_{x,y}$ does not.

\begin{proposition}
    For $z=0$, there exists a pair of singular links that are distinguished by $\Upsilon_{x,y,z}$, but it is not distinguished by $\Psi_{x,y}$.
\end{proposition}

 \begin{proof}
  Set $L_1=L11n358\{0, 1\}$ and $L_2=L11n418\{0, 0\}$ and consider pair of singular links in Figure \ref{slinks}
	
\begin{figure}[h!]\label{slinks}
\centering\includegraphics{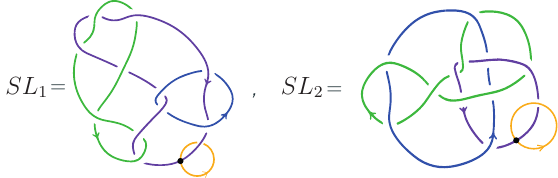}
\caption{Singular links $SL_1$ and $SL_2$.}
\end{figure}

	Thus, applying the desingularization rule to $SL_1$ we obtain
	\begin{figure}[h!]
		\centering
		\includegraphics{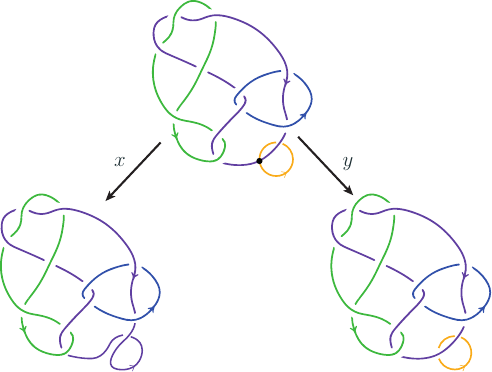}
	\end{figure}
where the link on the left is isotopic to $L_1$ and the other one is isotopic to $L_1\sqcup \bigcirc$.

	On the other hand, applying the desingularization rule to $SL_2$ we obtain
		\begin{figure}[h!]
		\centering
		\includegraphics{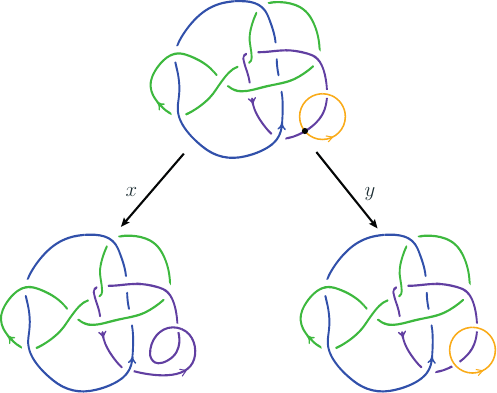}
	\end{figure}
where the link on the left is isotopic to $L_2$ and the other one is isotopic to $L_2\sqcup \bigcirc$.

	Then,  we have 
	\begin{align*}
		\Upsilon_{x,y,0}(SL_i)&=  x\Upsilon_{x,y,0}(L_i)+ y\Upsilon_{x,y,0}(L_i\sqcup \bigcirc)\\
		&=  x\Upsilon(L_i)+ y\Upsilon(L_i\sqcup \bigcirc) \\
		&= x\Upsilon(L_i)+ \frac{y}{aw}\Upsilon(L_i)\\
		&= (x+ \frac{y}{aw})\Upsilon(L_i)
	\end{align*}
for $i=1,2$. Analogously, we can obtain that $\Psi_{x,y}(SL_i)=(x+ \frac{y}{aw})\Delta(L_i)$, for $i=1,2$. And, we know that $\Delta(L_1)=\Delta(L_2)$ and $\Upsilon(L_1)\not=\Upsilon(L_2)$. Thus, the pair is distinguished by $\Upsilon_{x,y,0}$ but it is not distinguished by $\Psi_{x,y}$.   
 \end{proof}

\subsubsection{Further directions}

In our pursuit of understanding the implications of the parameter $z$ within the representation $\rho$ of Proposition \ref{newrep}, it becomes apparent that its inclusion is intended to improve the invariant $\widetilde{\Upsilon}_{x,y,z}$ in comparison to its $z=0$ counterpart, $\widetilde{\Upsilon}_{x,y,0}$. This improvement would come from the desingularization rule (\ref{desingtie}) in Theorem \ref{skeinth}. We believe that the introduction of the $z$-term into the desingularization process yields  cts--links that are not catpured by the invariant $\widetilde{\Upsilon}_{x,y,0}$.

However, this assertion needs a more comprehensive database of the invariant $\Upsilon$. Specifically, we need to find new pairs of links undifferentiated by $\Upsilon$ to potentially generate a pair of singular links that $\widetilde{\Upsilon}_{x,y,0}$ does not distinguish but $\widetilde{\Upsilon}_{x,y,z}$ does. While the theoretical framework hints at the potential discriminatory power of $\widetilde{\Upsilon}_{x,y,z}$ over $\widetilde{\Upsilon}_{x,y,0}$, empirical evidence remains elusive. Thus far, we have not found yet an example that proves this assertion. Nevertheless, we still believe that such examples exist. Hence, we propose the following conjecture:

\begin{conjecture}
    There exists a pair of cts--links that are distinguished by $\widetilde{\Upsilon}_{x,y,z}$ but remain indistinguishable under $\widetilde{\Upsilon}_{x,y,0}$.
\end{conjecture}

\section{A construction via the two parameter Singular bt--algebra}\label{sec:invariant singular bt}
In this section we recover the invariant defined in the previous section by following the approach by Paris and Rabenda in \cite{pare}. In \cite{aiju2018}, it is proved that the monoid $TSB_n$ is isomorphic to the semidirect product $\PP{n}\rtimes SB_n$ yielded by the natural action of $SB_n$ on $\PP{n}$. Subsequently, in \cite{arju2020}, it is introduced the concept of tied monoids. These monoids are defined as semi direct products that are commonly built from braid groups and their underlying Coxeter group acting on monoids of set partitions; and $TSB_n$ is a particular case of a tied monoid. Additionally, they propose a procedure to construct tied algebras from tied monoids. Moreover, the two parameter bt--algebra can be obtained by this procedure by using the monoid $TB_n$, see \cite[Section 7]{arju2020}. Thus, we introduced a new two parameter algebra that is obtained as a tied algebra of the monoid $TSB_n$. More precisely, we have the following definition.

\begin{definition}\label{SBTalg}
The \textit{two parameter singular bt--algebra}, denoted by $\SE{}{n}:=\SE{}{n}(u,v)$, is the $\KK$-algebra generated by 
$R_1,\ldots , R_{n-1}$, $S_1,\ldots , S_{n-1}$, $E_1,\ldots , E_{n-1}$ subject 
to relations (\ref{bt1})-(\ref{bt9}) together wi the following relations:

\begin{align}
    & E_i S_j =  S_j E_i \qquad \text{if $\vert i-j\vert \neq 1$},\label{Sbt1}\\
    & E_iS_jS_i= S_jS_iE_j \qquad \text{if $\vert i  -  j\vert =1$}, \label{Sbt2}\\
    & E_iE_jS_i  = E_j S_i E_j  =  S_iE_iE_j \qquad \text{if $\vert i  -  j\vert =1$}, \label{Sbt3}\\
    & E_iS_jR_i  = S_jR_iE_j \qquad\text{ if $\vert i-j\vert=1$,} \label{Sbt4}\\
    & E_iR_jS_i = R_jS_i E_j \qquad\text{ if $\vert i-j\vert=1$,} \label{Sbt5}\\
    & S_iS_j= S_j S_i \qquad \text{if $\vert i - j\vert \not= 1$}, \label{Sbt6}\\
    & S_iR_j= R_j S_i \qquad \text{if $\vert i - j\vert \not= 1$}, \label{Sbt7}\\
    & R_iR_jS_i = S_jR_iR_j \qquad \text{if $\vert i - j\vert = 1$}, \label{Sbt8}
\end{align}

\end{definition}

 \noindent The algebra $\SE{}{n}(u,v)$  can also be regarded as quotient of the monoid algebra $\KK(TSB_n)$ over the ideal generated by the elements $\sigma_i^2 - 1 -(u - 1)e_i - (v - 1)e_i\sigma_i$.

  \begin{remark}
      Let $\delta$ be a root of the polynomial $p({\sf x})=u({\sf x}+1)^2- (v-1)({\sf x}+1)-1$ and set $t:=u(\delta +1)^2$. Then, we have that $\mathcal{SE}_n(u,v)$ is isomorphic to the one parameter algebra $\mathcal{SE}_n(t,t)$ by \cite[Proposition 1]{aiju2020}.
  \end{remark}
  \begin{remark}
   For $u=v$, the mapping $R_i\mapsto h_i$ and $E_i\mapsto 1$ (resp. $R_i\mapsto g_i$, $E_i\mapsto 1$ and $S_i\mapsto p_i$) induces an homomorphism  $\phi: \E_n(u,u)\rightarrow H_n(u)$ (resp. $\psi: \SE{}{n}(u,u)\rightarrow SH_n(u)$). Then, we have the following commutative diagram
  
     $$\xymatrix{ \mathcal{E}_n(u,u)\ar@{^{(}->}[r] \ar_{\phi}[d] & \s{}{n}(u,u) \ar^{\psi}[d] \\  H_n(u) \ar@{^{(}->}[r] &  SH_n(u). }$$
  \end{remark}


\subsection{Markov traces on the two parameter Singular bt--algebra}\label{singularbt}
In this section, we construct a Markov trace on the algebra $\s{}{n}(u,v)$ by following the approach given in \cite{pare}. 
Let $\pi:TSB_n\rightarrow \SE{}{n}$ be the natural representation induced by the mapping $\sigma_i\mapsto R_i$, $\tau_i\mapsto S_i$ and $e_i\mapsto E_i$. For any subset $S\subset TSB_n$, $\pi(S)$ will be again denoted  by $S$, in particular, $\alpha\in TSB_n$ will denote $\pi(\alpha) \in \pi(TSB_n)$. 
%
%
For $d\geq 0$, let $TS_dB_n$ be the set of tied singular braids with $d$ singular points. Then, we set $\SE{d}{n}$ as the $\mathbb{K}$--linear subspace of $\s{}{n}$ generated by $TS_dB_n$. Thus, we have
$$\s{}{n}=\bigoplus_{d=0}^{\infty} \SE{d}{n}.$$
In particular, $\s{0}{n}$ is the two parameter bt--algebra $\mathcal{E}_n(u,v)$.\smallbreak

\begin{definition}
     For $d\geq 0$, a \textit{d-Markov trace} on $\{\SE{d}{n}\}_{n=1}^{\infty}$ is a family of $\mathbb{K}$--linear maps $\tr^d:=\{\tr_n^d\}_{n=0}^{\infty}$, where $\tr_n^d:\s{d}{n}\rightarrow \mathbb{K}(\A,\bb)$, satisfying the following
    \begin{enumerate}[(a)]
        \item $\tr_n^d(\alpha\beta)=\tr_n^d(\beta\alpha)$ for all $\alpha \in \TS{k}{n}$ and $\beta \in \TS{d-k}{n}$,
        \item $tr^d_{n+1}(\alpha)=\tr_n^d(\alpha)$ for all $\alpha \in \TS{d}{n}$,
        \item $\tr_{n+1}^d(\alpha R_n)=\tr_{n+1}^d(\alpha E_n R_n)= \A\tr_n^d(\alpha)$ for all $\alpha \in \TS{d}{n}$,
        \item $\tr_{n+1}^d(\alpha E_n)= \bb\tr_n^d(\alpha)$ for all $\alpha \in \TS{d}{n}$,
    \end{enumerate}
    for $n\geq 1$. Thus, a \textit{Markov trace} on $\{\s{}{n}\}_{n=1}^{\infty}$ is a family of $\mathbb{K}$--linear maps $\tr:=\{\tr_n\}_{n=0}^{\infty}$, such that $\{\tr_n^d:=\tr_n|_{\s{d}{n}}\}_{n=1}^{\infty}$ is a $d$-Markov trace for all $d\geq 0$.
\end{definition}

Set $X=\{0,1,-1\}$, and let be $\alpha=\alpha_0S_{i_1}\alpha_1S_{i_2}\cdots S_{i_{d}}\alpha_{d} \in \SE{d}{n}$, where $\alpha_l \in TB_n$. For $1\leq k \leq d$ and $r\in X$, we define the tied singular braid $\alpha_{r,k}\in \SE{d-1}{n}$ as follow:
$$\alpha_{r,k}:=
    \alpha_0 S_{i_1}\alpha_1\cdots S_{i_{k-1}}\alpha_{k-1} R_{i_k}^r\alpha_{k} S_{i_{k+1}}\cdots S_{i_{d}}\alpha_{d}, 
    $$
where by convention $R_{i_k}^0=1$ for any $k$.

For $r\in X$,  we define the map $\theta_r:TS_dB_n\to\s{d-1}{n}$ as follows
\begin{equation}
    \theta_r(\alpha) =  \sum_{k=1}^{d}\alpha_{r,k}.
\end{equation}
Thus, the map $\theta_r$ naturally extends to a $\KK$--linear map from $\s{d}{n}$ into $\s{d-1}{n}$, which will be also denoted by $ \theta_r$ for simplicity.\smallbreak 

From now on we work in the algebra $\s{}{n}(u,v)\otimes_{\KK} \mathbb{L}$, which will be denoted again by $\s{}{n}(u,v)$ for simplicity. Similarly, for $d\geq 0$ the $\mathbb{L}$--linear subspaces of $\s{}{n}(u,v)$ spanned by the tied singular braid with $d$ singularities will be denoted again by $\s{d}{n}$.\smallbreak

For $d\geq 0$, we define the $\mathbb{L}$--linear map $\tr^{(d)}:\s{d}{n}\rightarrow \mathbb{L}$ inductively as follows:
\begin{equation}\label{SEtrace}
    \tr^{(d)}(\alpha)=\tr^{(d-1)}(x\theta_0(\alpha)+yw\theta_1(\alpha)+zw^{-1}\theta_{-1}(\alpha)),
\end{equation}
where $\tr^{(0)}:=\{\Tr{n}\}_{n=1}^{\infty}$ is the Aicardi-Juyumaya's trace on the two parameter bt-algebra (see Section \ref{bttwo}). Note that for all $d\geq 0$ and $0\leq k \leq d$, $\tr^{(d)}:=\{\tr^{(d)}_n\}_{n=1}^{\infty}$ is a Markov trace on $\s{d}{n}$ since the Aicardi-Juyumaya's trace is as well.\smallbreak

We now introduce some notation that will be useful in the sequel.   Let be $\alpha=\alpha_0 S_{i_1}\alpha_1 S_{i_2}\cdots S_{i_{d}}\alpha_{d}$ as above and $u=(u_1,u_2,\ldots,u_d)\in X^d$. We define $u(\alpha)$ as the tied singular braid obtained by replacing $S_{i_k}$ by $R_{i_k}^{u_k}$. More precisely, we have
$$u(\alpha)=\alpha_0 R^{u_1}_{i_1}\alpha_1 R_{i_2}^{u_2}\dots R_{i_{d}}^{u_d}\alpha_{d}.$$

\noindent For $u=(u_1,u_2,\ldots,u_d)\in X^d$, we define the content of $u$, denoted by $\lambda_u$, as follows:
\begin{equation}\label{content}
   \lambda_u:=x^{e_0(u)}y^{e_1(u)}z^{e_{-1}(u)}w^{(e_1(u)-e_{-1}(u))}\in \mathbb{L},
\end{equation}
where $e_r(u)$ denotes the number of appearances of $r$ in $u$.\smallbreak

For instance, given $\alpha=R_1E_1S_1R_2S_2R_3S_3\in TS_3B_4$ and $u=(1,-1,0)\in X^{3}$, we have $u(\alpha)=R_1E_1 R_1R_3$ and $e_{-1}=e_1=e_0=1$. Then, the content of $u$ is given by $\lambda(u)=xyz$.

\begin{remark}\label{forth4}
    Observe that for $\alpha=\alpha_0S_{i_1}\alpha_1S_{i_2}\dots S_{i_{d}}\alpha_{d}\in TS_dB_n$ and $u=(u_1,u_2,\ldots,u_d)\in X^d$ we have that  
    $$u(\alpha)=\check{u}_k(\alpha_{u_k,k}) \qquad \text{for all $1\leq k \leq d$,} $$ 
  where $\check{u}_k$ is the element in $X^{d-1}$ obtained by deleting the $k$-th coordinate of $u$. Moreover, $\lambda_u=sg(u_k)\lambda_{\check{u}_k}$, where 
  $$sg(u_k):=\left\{\begin{array}{ll}
     x,  &  \text{if } u_k=0\\
     yw,  &  \text{if } u_k=1\\
     zw^{-1}, & \text{if } u_k=-1 
  \end{array}\right.$$
  for $1\leq k\leq d$.
\end{remark}

\begin{theorem}\label{tracesrel}\rm
    Let $\widetilde{\rho}:TSB_n\rightarrow \E_n(u,v)$ be the representation given in Proposition \ref{prop:ext}. Then, the following equality holds 
    \begin{equation}\label{eqtraces}
        (\tr^{(d)}\circ \pi)(\alpha)=d!(tr^{(0)}\circ \widetilde{\rho} )(\alpha)
    \end{equation}
    for $d\geq 0$ and $\alpha\in TS_dB_n$.
\end{theorem}
\begin{proof}Let be $\alpha=\alpha_0S_{i_1}\alpha_1S_{i_2}\dots S_{i_{d}}\alpha_{d}\in TS_dB_n$. First note that, for all $d\geq 0$, we have
  $$(\tr^{(0)}\circ  \widetilde{\rho} )(\alpha)=\sum_{u\in X^d} \tr^{(0)} (\lambda_u u(\alpha))$$ by multiplication principle.
  We now proceed the proof by induction on $d$. For $d=1$, Eq. (\ref{eqtraces}) holds trivially. By definition, for the general case we have
  \begin{align*}
      \tr^{(d)}(\alpha)&=\tr^{(d-1)}(x\theta_0(\alpha)+yw\theta_1(\alpha)+zw^{-1}\theta_{-1}(\alpha))\\
      &=\tr^{(d-1)}(x\theta_0(\alpha))+\tr^{(d-1)}(yw\theta_1(\alpha))+\tr^{(d-1)}(zw^{-1}\theta_{-1}(\alpha))\\
      &=\sum_{k=1}^d x\tr^{(d-1)}(\alpha_{0,k})+yw\tr^{(d-1)}(\alpha_{1,k})+zw^{-1}\tr^{(d-1)}(\alpha_{-1,k}).
  \end{align*}
  Applying induction hypothesis, we obtain
  {\small
  \begin{align}
      \tr^{(d)}(\alpha)&=(d-1)!\sum_{k=1}^d (tr^{(0)}\circ \widetilde{\rho} )(x\alpha_{0,k}+yw\alpha_{1,k}+zw^{-1}\alpha_{-1,k})\nonumber\\
      &=(d-1)!\left(\sum_{v\in X^{d-1}}\sum_{k=1}^d x\tr^{(0)}(\lambda_vv(\alpha_{0,k}))+yw\tr^{(0)}(\lambda_vv(\alpha_{1,k}))+zw^{-1}\tr^{(0)}(\lambda_vv(\alpha_{-1,k}))\right).\label{finalsum}
  \end{align}
  }
  By Remark \ref{forth4}, we have that for any $u\in X^d$, the element $\lambda_u u(\alpha)$ appears $d$ times in the sum into the brackets of Eq. (\ref{finalsum}). Therefore, we conclude that
$$\tr^{(d)}(\alpha)=(d-1)!\left(d\sum_{u\in X^d} tr^{(0)}(\lambda_u u(\alpha))\right) =d!(tr^{(0)}\circ  \widetilde{\rho} )(\alpha).$$
  
\end{proof}
Thus, for a cts--link $(L,I)$ we define:
\begin{equation}\label{paregen}
    \widehat{\Upsilon}_{x,y,z}(L,I):=\left(\frac{1}{\A w}\right)^{n-1}\left(\frac{\tr\circ \pi(\alpha)}{d!}\right),
\end{equation}
where $\alpha$ is a tied singular braid such that $\hat{\alpha}=(L,I)$. Then, we have that the function $\widehat{\Upsilon}$ recovers the cts--links invariant $\widetilde{\Upsilon}_{x,y,z}$ from Section \ref{newinv}.

\begin{remark}
Consider $u=v$ and $z=0$, and let be $\tau:\s{}{n}\rightarrow H_n(u)$ the homomorphism induced by the map $R_i\mapsto h_i$ and $E_i,S_i\mapsto 1$. Then, $\Pi_{x,y}:=\tau\circ \rho_{x,y,0}$ is a representation of $SB_n$ into the Iwahori--Hecke algebra $H_n$.
Thus, the Paris Rabenda invariant can be obtained in terms of the representation $\Pi_{x,y}:=\tau\circ \rho_{x,y,0}$ and the Ocneanu's trace $\mathsf{tr}$. In fact, by setting $X=y$ and $Y=x$, we have
\begin{equation*}
    \widehat{I}(L)=\left(\frac{1}{\A w}\right)^{n-1}\left(\frac{{\sf tr}\circ \Pi(\alpha)}{d!}\right),
\end{equation*}
where $L=\hat{\alpha}$, $\alpha\in SB_n$.
\end{remark}

\subsection*{Acknowledgements}
The first author was partially supported by grant UVA22991 (Proyecto PUENTE UV).

\bibliography{sbt}{}
\bibliographystyle{acm}

\bigskip

\noindent 
{\small \sc Instituto de Matemáticas, Universidad de Valparaíso,\\
Gran Breta\~{n}a 1111, Valpara\'{i}so 2340000, Chile.\\
{\tt marcelo.flores@uv.cl}\\
}

\smallskip
\noindent
{\small \sc Centro de Investigaci\'on y Docencia Econ\'omicas A.C.,\\
Carretera M\'exico-Toluca No. 3655, Ciudad de M\'exico 01210, M\'exico.\\
{\tt christopher.roque@cide.edu}}

\end{document}